\documentclass[12pt,reqno]{amsart}

\usepackage[T1]{fontenc}
\pdfoutput=1
\usepackage{mathptmx,microtype,graphicx,dsfont,booktabs}
\usepackage[hidelinks]{hyperref}
\usepackage{amsthm,amsmath,amssymb}
\usepackage{enumitem}
\usepackage[nameinlink,capitalise]{cleveref}
\usepackage[font=footnotesize,margin=5em]{caption}
\usepackage{cite}

\usepackage[foot]{amsaddr}

\theoremstyle{plain}
\newtheorem{thm}{Theorem}[section]

\newtheorem{prop}[thm]{Proposition}

\newcommand\Z{\mathbb{Z}}

\renewcommand{\P}{{\mathbb P}}

\renewcommand{\le}{\leqslant}
\renewcommand{\ge}{\geqslant}

\author[B. Kolesnik]{Brett Kolesnik}
\address{University of Warwick, Department of Statistics}
\email{brett.kolesnik@warwick.ac.uk}

\author[G. Zakharov]{Georgii Zakharov}
\address{University of Oxford, Department of Mathematics
and Exeter College}
\email{georgii.zakharov@exeter.ox.ac.uk}

\author[M. Zhukovskii]{Maksim Zhukovskii}
\address{University of Sheffield, School of Computer Science}
\email{m.zhukovskii@sheffield.ac.uk}


\begin{document}

\title[Threshold for triangulations 
inside convex polygons]
{On the threshold for triangulations 
inside convex polygons}

\begin{abstract} 
Start with a large convex polygon and 
add all other edges inside 
independently with probability $p$. 
At what critical threshold $p_c$ do triangulations
of the polygon begin to 
appear?

The first author and Gravner asked this question, 
and observed 
that $p_c=\Theta(1)$, using 
the relationship with the Catalan numbers  
and a coupling with 
oriented site percolation on $\Z^2$. 
More recently, Archer, Hartarsky, the first author, 
Olesker-Taylor, Schapira and Valesin
proved that $1/4<p_c<p_c^o$, where
$1/4$ is the Catalan exponential growth rate 
and $p_c^o$ is the critical threshold for 
oriented percolation. 
The upper bound is strict, but non-quantitative,
and follows by a renormalization argument. 

We show that $p_c<1/2$ 
using  
a simple ear clipping algorithm, which can be analyzed  
using the gambler's ruin problem. 
This bound is closer to the truth (perhaps near $0.4$)
and shows that most configurations
of edges inside large convex polygons contain triangulations. 
\end{abstract}

\maketitle

\section{Introduction}\label{S_intro}

Let $P_n$ be a convex polygon
with vertices labeled by $\{1,\ldots,n\}$.
We further include a set $E_{n,p}$ of random edges, obtained
by adding all other edges inside $P_n$
independently with probability $p$. 
We are interested in the critical point $p_c$ 
at which triangulations
of $P_n$ appear. More formally, 
if $\mathcal{T}_{n,p}$ is the event that 
$P_n$ can be 
triangulated using the edges of $E_{n,p}$, then  
\[
p_c
=\inf\{p:\liminf_{n\to\infty}\P(\mathcal{T}_{n,p})>0\}
\]
is the critical threshold at which 
triangulations 
begin to emerge. 

Equivalently, we can arrange $n$ distinct points along a circle 
with the curves between adjacent pairs becoming edges,
and then add an Erd{\H o}s--R\'enyi graph ${\mathcal G}_{n,p}$ on top of this. 
From this perspective, $p_c$ is simply the point at which 
${\mathcal G}_{n,p}$ triangulates 
the points along the circle. 

In this work, we show the following.

\begin{thm}\label{T_mainUB}
There exists $\varepsilon>0$ such that 
 $p_c<1/2-\varepsilon$. 
\end{thm}

We actually show that $p_c\le p_*$, where $p_*\approx 0.4916$. 
Theorem~\ref{T_mainUB} implies that
most configurations of edges
inside a large convex polygon
can be used to triangulate the polygon.

\subsection{Previous results}

Recent work by  
Archer, Hartarsky, the first author, Olesker-Taylor, 
Schapira and Valesin \cite{AHKOTSV25} 
shows that  
$p_c<p_c^o$,  where $p_c^o$ is the critical threshold
for oriented site percolation on the integer lattice $\Z^2$; see, e.g.,
Durrett \cite{Dur84}. 
Numerical simulations indicate that $p_c^o\approx .7055$; see
Essam, Guttmann and De'Bell
\cite{EGDeB88}. On the other hand, the numerics
in \cite{AHKOTSV25} suggest that $p_c$ is, in fact, much smaller,  
perhaps somewhere between 0.39 and 0.41; see \cite[Fig.\ 3]{AHKOTSV25}.

The connection with oriented
percolation was already observed by 
Gravner and the first author \cite{GK23} 
(see Theorem 1.3, Section 3 and Conjecture 6.1), 
where the problem of finding $p_c$ was referred to as 
{\it Catalan percolation},
as a special case of the transitive closure
dynamics in polluted environments studied therein (see Section \ref{S_equiv} below). 
The motivation in \cite{GK23} was to 
bring together ideas from {\it weak saturation} (see, e.g., 
Bollob\'{a}s \cite{Bol68}
and 
Balogh, Bollob\'{a}s and Morris \cite{BBM12})
and {\it polluted bootstrap percolation} 
(see, e.g., Gravner and McDonald \cite{GMcD97}),
in response to the final paragraph in 
\cite[p.\ 439]{BBM12}.

Let us remark that, 
although Theorem \ref{T_mainUB}  
improves on the upper bound in \cite{AHKOTSV25},
it is, in fact, the {\it proof}
rather than the {\it result}
that is the main contribution in \cite{AHKOTSV25}. 
Indeed,   
the coupling, first observed in \cite{GK23},  
with oriented percolation  
is based on a significant
restriction of the full Catalan dynamics 
(see Section \ref{S_equiv} below). 
As such, $p_c<p_c^o$ is certainly not unexpected. 
However, due to long-range, non-decaying correlations
in the model, it is not straightforward to
deduce a {\it strict} inequality using standard
techniques 
from percolation
(e.g., the method of 
{\it essential enhancements} 
of Aizenman and Grimmett \cite{AG91} does not apply); 
see \cite[\S1.2]{AHKOTSV25}
for a detailed discussion. 
As such, ideas in the proof in \cite{AHKOTSV25}
may be useful in analyzing other oriented percolation models, 
beyond the specific case of Catalan percolation. 
Indeed, the dynamics studied in \cite{AHKOTSV25} 
can be thought of as a 
directed version of 
{\it Brochette percolation}, as 
studied by Duminil-Copin, Hil\'{a}rio, Kozma and Sidoravicius
\cite{DCHKS18}.

\subsection{Equivalence with Catalan percolation}
\label{S_equiv}

The perspective taken 
in \cite{GK23,AHKOTSV25}
is conceptually different than ours, 
but is formally equivalent,
as we will now explain. 

In \cite{GK23} the authors considered
the following situation: start by initially
{\it infecting} all directed nearest-neighbor
edges along the integer line path from 
$1$ to $n$. Then {\it open} all other leftward
(resp.\ rightward) directed edges $i\leftarrow j$ 
(resp.\ $i\rightarrow j$), for $1\le i<j\le n$ and $j-i>1$, 
with some probability 
$p_\ell$ (resp.\ $p_r$). 
All other directed edges are {\it polluted} 
and can never become infected. 
Open edges, on the other hand, can become infected 
by the following transitive closure dynamics: 
if at some point there are 
directed edges $i\to j\to k$
which 
are both infected (initially or otherwise) 
then $i\to k$ becomes infected if it is open. 
This model is introduced in 
\cite{GK23} as a simple model 
for the spread of information 
in the presence of censorship. 
(In fact, other initial graphs, 
besides paths, are considered in 
\cite{GK23}.)

What is called {\it Catalan percolation}
in \cite{GK23} is the special case 
in which $p_\ell=0$ and $p_r>0$. 
(When both $p_\ell,p_r>0$ the behavior is very different,
and still not fully understood.) 
In this case the dynamics have a simple
graphical description, using undirected edges, 
in which open edges $\{i,k\}$ become infected
if there are two infected edges $\{i,j\}$ and $\{j,k\}$
``underneath'' the edge $\{i,k\}$. 
In \cite{AHKOTSV25}, the authors
let $\varphi_n(p)$ be the probability that
the edge 
$\{1,n\}$ joining the endpoints of the path 
from $1$ to $n$ 
is eventually infected, conditional on it being open,
and put $p_c=\inf\{p:\liminf_{n\to\infty}\varphi_n(p)>0\}$.
However, the eventual infection of $\{1,n\}$
by these dynamics, assuming 
that it is open, is equivalent to the existence of 
a triangulation of $P_n$ using the edges in $E_{n,p}$;
see Figure \ref{F_cat} below. 

\begin{figure}[h!]
\centering
\includegraphics[scale=1]{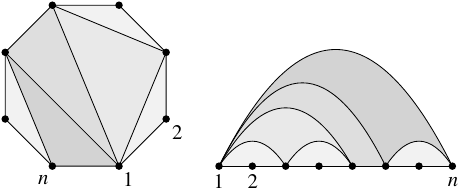}
\caption{Comparing Catalan percolation 
to the existence of a triangulation.}
\label{F_cat}
\end{figure}

The Catalan numbers are, of course,
related to many combinatorial objections. 
In \cite{GK23,AHKOTSV25} the 
connection is, in some sense, 
viewed in terms of parenthesizations of a product. 
In the current work, we work with triangulations instead. 

The coupling with oriented percolation, used in 
\cite{GK23,AHKOTSV25}, 
arises from a simple restriction of the Catalan percolation 
dynamics, where $\{i,k\}$ is only infected
if also at least one of the infected
edges $\{i,j\}$ or $\{j,k\}$ underneath it is a nearest-neighbor
edge.

\subsection{Our strategy}

In this work, we recast Catalan percolation
in a new light, in terms of triangulations, 
which leads to better bounds and simpler proofs, using  
a natural connection with 
random walks. 

The proof of Theorem \ref{T_mainUB} 
gives an upper bound of about $0.4916$, 
and elaborations on our methods 
would yield further 
improvements. 
In fact, our method of proof is rather straightforward
and elementary. We simply note that 
many natural ear clipping algorithms, 
using the random edges $E_{n,p}$,  
can be analyzed by the
gambler's ruin problem, and variations 
thereof, as discussed, e.g., in 
Feller 
\cite[XIV.\ 8]{Fel68} (see Section \ref{S_Feller}
below). 
We recall that an {\it ear} in a triangulation
of a polygon 
is a triangle with two adjacent edges 
along the boundary of the polygon.
Removing such a triangle leaves us with a 
smaller convex polygon
to triangulate. 

More specifically, we will first
show (see Sections \ref{S_GECA} and \ref{S_GTA} 
below) that the simplest greedy
ear clipping algorithm quite naturally
leads to a direct comparison
with the classical gambler's ruin problem. 
This gives a very short proof that 
$p_c\le 1/2$, 
along with a linear time algorithm of finding a triangulation
when it exists. 
Then (in Section \ref{S_BECA} below), 
to establish a strict inequality 
$p_c<1/2$, we relax the greedy dynamics, 
and allow  additional flexibility
at the local decision level. 
Since this leads, once again, to a Markov chain with negative
drift, a similar proof applies (using the general bounds 
recalled in Section \ref{S_Feller} below). 
Expanding this local neighborhood further would 
lead to further improvements, 
but, as it would seem, at the expense of increasingly 
more elaborate arguments.

\subsection{Ear clipping comparison}
Finally, to further compare
with the previous works \cite{GK23,AHKOTSV25} 
discussed above, 
let us note that the 
proof of the 
oriented percolation bound can also be viewed 
in terms of a certain ear clipping strategy. 
This strategy, however, is much more restricted, 
requiring that each ear being clipped contains at 
least one edge in the original polygon $P_n$. 

Basically, in the coupling with oriented percolation, 
the orientation of a step corresponds to either clipping directly  
to the left or right of the currently clipped region;
see, e.g., 
\cite[Fig.\ 3]{GK23} and 
\cite[Fig.\ 6]{AHKOTSV25}. 
The strict inequality
in \cite{AHKOTSV25}
essentially comes from studying 
a process that can 
clip slightly deeper into the polygon, by also clipping ears of $P_n$ 
if they are adjacent to the currently clipped region; 
see \cite[Fig.\ 8]{AHKOTSV25}.

In this work, on the other hand, we consider exploration 
processes that can clip well into the polygon, 
and this opens up a lot more combinatorial freedom;
see, e.g., Figure \ref{F_GTA} below. 
Fortunately, although 
these processes are less restricted, there
is a comparison with 
reflected biased random walks to be made, 
leading to relatively simple proofs.

\subsection{Related works}

In closing, let us mention
a few other works in the literature
with some connection to ours. 
First, we recall that 
Bollob\'{a}s and Frieze \cite{BF91}
found the thresholds for spanning maximal planar/outerplanar
subgraphs of the random graph 
${\mathcal G}_{n,p}$. In particular, their result implies the 
threshold probability, up to a polylogarithmic factor, 
for the appearance of a triangulation of a triangle --- i.e., a triangulation with the maximum possible number of ``internal'' vertices ---
 and a triangulation of an $n$-gon --- a somewhat opposite 
 case when there are no internal vertices in a triangulation. 
 A precise order of magnitude for the former threshold 
 follows from a later result of Riordan~\cite{Riordan}. 
 The threshold for the latter case follows, up to a constant factor, 
 from the recent results on the square of a Hamiltonian cycle by 
Kahn, Narayanan and Park \cite{KNP21}
and the third author \cite{Zhu25}.

\section{Gambler's ruin}
\label{S_Feller}

We recall a result   
on the generalized gambler's ruin problem for
random walks $(X_t)$ on $\Z$
with bounded jumps; see, e.g., 
Feller \cite[XIV.\ 8]{Fel68} on what is called
{\it sequential sampling} therein. 

Consider a random walk (i.e., a time- and space-homogeneous Markov chain) 
$(X_t)_{t\in\mathbb{Z}_{\geq 0}}$ on $\Z$ 
starting at some $X_0=x$ with $0<x<J$, 
where $J$ is some fixed integer that we call the {\it jackpot}. 
We let  
\[
p(k)=\P(X_{t+1}-X_t=k) 
\]
denote its transition probabilities. 
 We also suppose that, for some integers $\nu,\mu>0$, 
 we have that 
\begin{itemize}
\item $p(-\nu)>0$;
\item $p(\mu)>0$; and 
\item $p(k)=0$
for $k<-\nu$ or $k>\mu$. 
\end{itemize}
In other words $[-\nu,\mu]$ is the smallest interval containing
the support of $p$. 
Furthermore, we suppose that $(X_t)$ has  
a negative drift
\[
 \Delta=\sum_k kp(k)<0.
\] 

In Feller \cite[XIV.\ 8, (8.12)]{Fel68}, bounds 
are given for the probability 
\[
\phi(x,J)
=
\P(\inf\{t:X_t\ge J\}<\inf\{t:X_t\le 0\})
\] 
that $(X_n)$ reaches a jackpot value $\ge J$
before a {\it ruin} value $\le 0$. 
Specifically,  
\begin{equation}\label{E_Feller}
\frac{\alpha_*^x-1}{\alpha_*^{J+\mu-1}-1}
\le \phi(x,J)
\le 
\frac{\alpha_*^{x+\nu -1}-1}{\alpha_*^{J+\nu -1}-1},
\end{equation}
where $\alpha_*>1$ is the unique $\alpha\neq1$ satisfying 
the characteristic equation 
\[
\sum_k \alpha^k p(k)=1.
\]
In our applications of \eqref{E_Feller}, we will have 
$x$ fixed and 
$J=\delta\log n$, in which case, 
as $n\to\infty$, 
\begin{equation}\label{E_UBell}
\phi(x,J)=O(n^{-\delta\log\alpha_*}). 
\end{equation}

Finally, let us note that, if $\nu=\mu=1$, 
then \eqref{E_Feller} reduces to the classical 
gambler's ruin formula
\begin{equation}\label{E_GR}
\phi(x,J)
=
\frac{(p/q)^x-1}{(p/q)^J-1},
\end{equation}
with $p:=p(-1)>p(1)=:q$.

\section{Greedy ear clipping}
\label{S_GECA}

In this section, we describe a natural {\it greedy 
ear clipping algorithm (GECA)}, which 
will play an important role in our proofs. 

{\bf Input.} As input, GECA takes in: 
\begin{itemize}
\item a polygon $P$; 
\item with vertices labelled by $v_1,\ldots,v_n$ 
in counter-clockwise order
around the boundary of $P$; and 
\item some set of edges $E$ inside $P$.
\end{itemize}

{\bf Output.} If successful, GECA returns:
\begin{itemize}
\item a polygon $P'$; 
\item with vertices $v_1,v_\tau,\ldots,v_n$, 
for some $3\le \tau \le n$; and 
\item a triangulation of the region inside $P$
{\it to the right} of $\{v_1,v_\tau\}\in E$ 
(i.e., the region is bounded by a polygon 
with vertices $v_1,v_2,\ldots,v_{\tau}$).
\end{itemize}

{\bf Algorithm.} In the $k$th step of GECA we have
a polygon $P^{(k)}$ and a list $\ell_k$ of vertices 
along a counter-clockwise path along 
the boundary of the polygon starting from $v_1$. 
We start with 
$P^{(0)}=P$
and $\ell_0=(v_1,v_2,v_3)$. 
Suppose that after the $k$th step of the algorithm, 
we have 
a polygon $P^{(k)}$ and a 
list 
\[
\ell_k=(v_1=v_1^{(k)},\ldots,v_{m_k}^{(k)}).
\]
Then, in the $(k+1)$th step of GECA, we proceed
as follows: 
\begin{itemize}
\item {\bf Clipping step:} If $\{v_{m_k-2}^{(k)},v_{m_k}^{(k)}\}\in E$, 
then we obtain $P^{(k+1)}$
from $P^{(k)}$ by clipping the ear   
induced by $\{v_{m_k-2}^{(k)},v_{m_k-1}^{(k)},v_{m_k}^{(k)}\}$. 
In this case, we put 
\[
\ell_{k+1}=(v_1^{(k)},\ldots,v_{m_k-2}^{(k)},v_{m_k}^{(k)}). 
\]
\item {\bf Extending step:} Otherwise, if 
$\{v_{m_k-2}^{(k)},v_{m_k}^{(k)}\}\notin E$, let  
$P^{(k+1)}=P^{(k)}$
and 
\[
\ell_{k+1}=(v_1^{(k)},\ldots,v_{m_k}^{(k)},v), 
\]
where $v$ is the next vertex 
after $v_{m_k}^{(k)}$
in counter-clockwise order
along the boundary of $P$. 
\end{itemize}

See Figure \ref{F_GECA}
for an illustration. 

\begin{figure}[h!]
\centering
\includegraphics[scale=1]{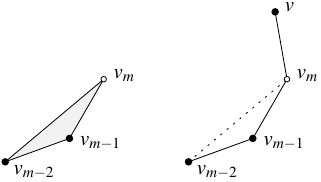}
\caption{In GECA,
we clip an ear if possible, and otherwise 
move to the next vertex along the boundary of $P$.}
\label{F_GECA}
\end{figure}

{\bf Termination.} 
GECA continues in this way, until eventually 
in some step $k$ 
either:
\begin{itemize}
\item {\bf Success:} $\ell_k=(v_1,v_\tau)$, 
for some $3\le \tau\le n$, in which case 
(by induction) GECA has triangulated the region 
inside $P$ to the right of
$\{v_1,v_\tau\}\in E$, and we terminate 
the algorithm and return 
$P'=P_k$; or else 
\item {\bf Failure:} the last vertex in the list 
$\ell_k$ is $v_{m_k}^{(k)}=n$
and 
$\{v_{m_k-2}^{(k)},v_{m_k}^{(k)}\}\notin E$, 
in which case we say that GECA has failed
and terminate the algorithm. 
\end{itemize}

\section{Greedy triangulations}
\label{S_GTA}

We show that
if $p>1/2$ then 
we can 
find a triangulation of $P_n$
using a certain {\it greedy triangulation algorithm (GTA)}
based on GECA.  

\begin{prop}
\label{P_p12}
For every constant $p>1/2$, with high probability, 
we can find a triangulation of $P_n$ 
in linear time. 
In particular,  $p_c \le 1/2$. 
\end{prop}

The proof is almost as simple as iterating
the GECA, and applying the classical 
gamber's ruin formula \eqref{E_GR}. However, 
as we will see, some care is required
as we near the end our of
tour of the boundary
of $P_n$. 

To overcome this issue, we identify a set 
$B=\{n-b,\ldots,n\}$
of vertices, which we will call the {\it buffer}. 
Here $b= \beta\log n$ (ignoring insignificant rounding issues, 
here
and throughout this work) where
$\beta>0$ is a small positive constant, to be determined below. 
Then, 
roughly speaking, to run GTA we will proceed as follows: 

\begin{itemize}
\item {\bf Root finding:} We iterate GECA. 
In the first application of GECA, we start with  
$P^{(0)}=P_n$, $E=E_{n,p}$, and $\ell_0=(1,2,3)$. 
Each subsequent application of GECA is applied 
to the polygon that the previous application of GECA outputs.
Recall that, after the $k$th application, 
we will have triangulated the region inside of $P_n$ 
to the right of some edge $\{1,v_\tau\}\in E_{n,p}$. 
We continue to iterate GECA until the first time that 
 $v_\tau$ is neighbors with all vertices in the buffer $B$. 
We denote this vertex by $\rho$, and call it the 
{\it root}. We let $P_\rho$ denote the polygon delimited by 
$\{1,\rho\}$ and the path in $P_n$ to 
the left of this edge. 
\item {\bf Completion:} Once we have found $\rho$, 
we continue to iterate starting with $P_\rho$ and first vertex $v_1:=\rho$. 
More precisely, 
in the first iteration, we 
begin with $P^{(0)}=P_\rho$, $E=E_{n,p}$ 
(or, more precisely, the edges of $E_{n,p}$ inside $P_\rho$), 
and $\ell_0=(\rho,u,v)$, 
where $u$ and $v$ are the next two vertices after $\rho$ in 
(the counter-clockwise path around the boundary of) $P_n$. 
We continue to iterate until the first time that some
iteration of GECA finishes with some edge
$\{\rho,u_k\}$ with $u_k\in B$ in the buffer. 
At this point, we halt GECA, and complete the triangulation 
of $P_n$ using edges between $\rho$ and $B$. 
\end{itemize}

Of course, in the proof below, we will need to show that,
with high probability, this procedure is well defined. 

See Figure \ref{F_GTA}
for an example. 

\begin{figure}[h!]
\centering
\includegraphics[scale=1]{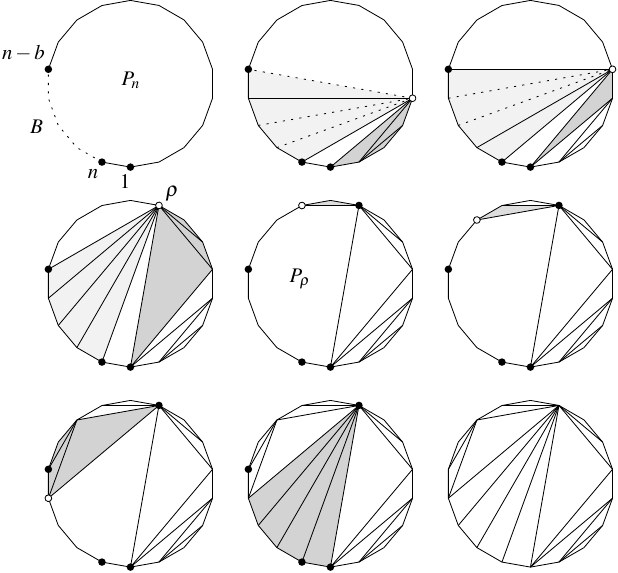}
\caption{Using GTA to triangulate $P_n$:  
Iterations of GECA are shown
in dark grey. The light grey regions depict attempts 
at finding a root. 
After the third such attempt, we find $\rho$. We then
continue to triangulate the remaining polygon $P_\rho$, 
until some iteration of GECA ends in the buffer $B$. 
Finally, we  
complete
the triangulation using 
 edges between $\rho$ and $B$.}
\label{F_GTA}
\end{figure}

\newpage
\begin{proof}[Proof of Proposition \ref{P_p12}]
We will use GTA, as described above. 

The key observation is to note that GECA 
defines a simple random walk, with a negative drift.
Indeed, recall that GECA starts with a list of length 3. 
In clipping steps the list decreases by 1, and in extending steps
it increases by 1. It is easy to see that these steps  
occur with probabilities $p$ and $q=1-p$, respectively, 
as they do not depend on 
edges that affected the outcome of any previous steps. 

As such, the gambler's ruin formula \eqref{E_GR}
applies. In particular, the probability that, in any iteration 
of GECA used while running GTA, some list 
reaches length 
$J=\delta\log n$
is at most
\[
O(n(p/q)^{-J})=O(n^{1-\delta\log(p/q)})\ll1,
\]
for a sufficiently large $\delta>0$. 

Hence, for any such $\delta>0$, there will 
be, with high probability, $\Omega(n/\log n)$ many opportunities
to find a root $\rho$. Note that, at each such 
opportunity, we successfully find $\rho$ 
with probability 
$p^{b+1}$, where recall that $b=\beta \log n$. Note that
\[
\frac{n}{\log n} p^{b+1} = \Omega\left(n^{1+\beta\log p}/\log n\right)\gg 1,
\]
for any small $\beta>0$. 

Therefore, for any large
$\delta>0$ and small $\beta>0$, we will, with high probability, 
find a root $\rho$ somewhere along the first half 
(in counter-clockwise order starting from $1$ of the 
boundary) of $P_n$. 

Assuming that 
$\rho$ has been found and that all lists in all iterations of 
GECA never exceed length $J=\delta \log n$, then there will be some 
iteration of GECA which ends with an edge $\{\rho,u\}$
for some $(n-b)-2J\le u\le (n-b)-J$.
However, a simple union bound over at most $2J+b=O(\log n)$ 
iterations of GECA (using the Markov property of the process) 
further shows that, with high 
probability, no subsequent iterations of GECA will have a list 
whose length ever 
exceeds $b/2$. Therefore, some final iteration of GECA
will terminate with an edge $\{\rho,v\}$ with $v\in B$. 
Finally, we complete the triangulation of $P_n$
using the edges between $\rho$ and the rest of the 
vertices along $P_n$ between $v$ and $n$. 
\end{proof}

\section{Better clipping}
\label{S_BECA}

Finally, we prove our main result 
Theorem \ref{T_mainUB}. 
The proof is similar in spirit to 
that of Proposition \ref{P_p12}.
However, instead of proceeding greedily via GECA, 
we will be more judicial about ear clipping, 
using a certain {\it better ear clipping algorithm (BECA)},
as depicted in Figure \ref{F_BECA} below. 
In contrast to GECA, this algorithm initiates with a list 
$\ell_0$ of length 4. At each iteration, the algorithm starts 
from the clipping step: it reveals some adjacencies 
in a certain order and, 
depending on the revealment, does one of 7 moves, as  
in Figure \ref{F_BECA}. At the end of the clipping step in 
a single iteration of BECA, if the list reduces to length 2 or 3, 
we proceed as in GECA: we apply the extending step, 
appending the next vertices of the polygon so that, 
in the next iteration, the list has length exactly 4.

\begin{figure}[h!]
\centering
\includegraphics[scale=1]{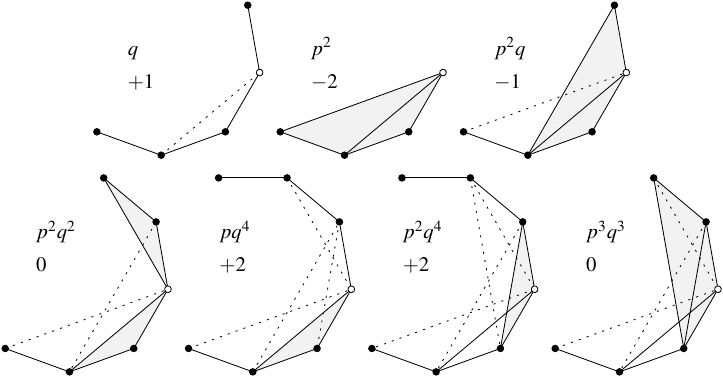}
\caption{In BECA, before clipping an ear, 
we consider the 
benefit going forward. 
In each step, we take one of the above 
mutually exclusive moves.
To determine which move is made, we reveal some edges 
(and non-edges) on a need-to-know basis. 
This can be done by a simple search algorithm, 
as indicated by the order in which the moves are listed above. 
As in Figure \ref{F_GECA} for the GECA, 
the open dot represents the current
end of the list $\ell$. The vertex along the 
path furthest to its right (possibly itself) 
will be at the end of the list in the next step. 
The change in the length of the list
is indicated.}
\label{F_BECA}
\end{figure}

\begin{proof}[Proof of Theorem \ref{T_mainUB}]
In a nutshell, 
we simply apply the proof of Proposition \ref{P_p12}, 
but 
replace the role of GECA with that of BECA, 
and 
apply the general bounds \eqref{E_Feller} 
for $\phi$ rather than the basic formula \eqref{E_GR}. 
Note that, in order to make a 
BECA step we require
the current list to have length at least~4. 
Moreover, one iteration in the BECA algorithm may 
consume up to three vertices to the right of the last vertex $v_i$ of the list. 
Therefore,  $v_i$ should always satisfy $i\le n-3$.  
With high probability this always happens due to the 
logarithmic size of the buffer --- the last application of BECA 
outputs a list whose last vertex is $\Theta(\log n)$-far from $v_n$. 

Let us proceed with the details, in the few places where 
they differ from the proof of 
Proposition \ref{P_p12}. 

First, let us note that, when using BECA,
we still have a Markov chain. 
Indeed, by Figure \ref{F_BECA}, 
we see that by using a simple search algorithm, 
revealing edges in $E_{n,p}$ on a need-to-know
basis, we can determine which move
to make in a BECA step. 
Furthermore, assuming $v:=v_{m_{k-1}}^{(k-1)}$ and 
$u:=v_{m_{k}}^{(k)}$ are the last vertices in the lists 
$\ell_{k-1}$ and $\ell_{k}$ after the $(k-1)$th and $k$th steps of BECA, 
we have that the $k$th step depends only on adjacencies in $E_{n,p}$ 
between the set vertices to the right of $v$ but to the left of $u$ and the 
set vertices to the right of $u$ and including $u$ itself. 
Therefore, all the adjacencies that BECA reveals in these two 
consecutive steps are disjoint, implying the desired Markov property. 

Finally,
note that BECA steps have a negative drift. 
Indeed, from Figure \ref{F_BECA}, we see that 
the expected change $\Delta$ in the length of a list $\ell$
after a BECA step  
(where $q=1-p$) is 
\begin{align*}
\Delta
&=
q
-2p^2
-p^2q
+2pq^4
+2p^2q^4\\
&=
2p^6
-6p^5 
+4p^4
+5p^3 
-9p^2
+p 
+1. 
\end{align*}
We note that $\Delta<0$
for all $p>p_*$, where
$p_*\approx 0.4916$. 

Therefore, for any such constant $p>p_*$, \eqref{E_UBell}
applies with some $\alpha_*>1$, 
and so we can complete the proof along the same lines as
Proposition \ref{P_p12}. 
\end{proof}

\section*{Acknowledgments}
This work began while GZ and MZ were visiting BK at 
the University of Warwick. We thank 
the Centre for Research in Statistical Methodology (CRiSM) 
for funding this visit.


\begin{thebibliography}{10}

\bibitem{AG91}
M.~Aizenman and G.~Grimmett, \emph{Strict monotonicity for critical points in
  percolation and ferromagnetic models}, J. Statist. Phys. \textbf{63} (1991),
  no.~5-6, 817--835.

\bibitem{AHKOTSV25}
E.~Archer, I.~Hartarsky, B.~Kolesnik, S.~Olesker-Taylor, B.~Schapira, and
  D.~Valesin, \emph{Catalan percolation}, Probab. Theory Related Fields,
  Special Issue: In Celebration of Geoffrey Grimmett's 70th Birthday, to
  appear, preprint available at
  \href{https://arxiv.org/abs/2404.19583}{arXiv:2404.19583}.

\bibitem{BBM12}
J.~Balogh, B.~Bollob\'{a}s, and R.~Morris, \emph{Graph bootstrap percolation},
  Random Structures Algorithms \textbf{41} (2012), no.~4, 413--440.

\bibitem{Bol68}
B.~Bollob\'{a}s, \emph{Weakly {$k$}-saturated graphs}, Beitr\"{a}ge zur
  {G}raphentheorie ({K}olloquium, {M}anebach, 1967), B. G. Teubner
  Verlagsgesellschaft, Leipzig, 1968, pp.~25--31.

\bibitem{BF91}
B.~Bollob\'{a}s and A.~M. Frieze, \emph{Spanning maximal planar subgraphs of
  random graphs}, Random Structures Algorithms \textbf{2} (1991), no.~2,
  225--231.

\bibitem{DCHKS18}
Duminil-Copin, H., Hil\'{a}rio, M. R., Kozma, G.,  
and Sidoravicius, V., \emph{Brochette percolation}, 
Israel J. Math. \textbf{225} (2018), no.~1, 479--501.

\bibitem{Dur84}
R.~Durrett, \emph{Oriented percolation in two dimensions}, Ann. Probab.
  \textbf{12} (1984), no.~4, 999--1040.

\bibitem{EGDeB88}
J.~W. Essam, A.~J. Guttmann, and K.~De'Bell, \emph{On two-dimensional directed
  percolation}, J. Phys. A \textbf{21} (1988), no.~19, 3815--3832.

\bibitem{Fel68}
W.~Feller, \emph{An introduction to probability theory and its applications.
  {V}ol. {I}}, third ed., John Wiley \& Sons, Inc., New York-London-Sydney,
  1968.

\bibitem{GK23}
J.~Gravner and B.~Kolesnik, \emph{Transitive closure in a polluted
  environment}, Ann. Appl. Probab. \textbf{33} (2023), no.~1, 107--126.

\bibitem{GMcD97}
J.~Gravner and E.~McDonald, \emph{Bootstrap percolation in a polluted
  environment}, J. Statist. Phys. \textbf{87} (1997), no.~3-4, 915--927.

\bibitem{KNP21}
J.~Kahn, B.~Narayanan, and J.~Park, \emph{The threshold for the square of a
  {H}amilton cycle}, Proc. Amer. Math. Soc. \textbf{149} (2021), no.~8,
  3201--3208.

\bibitem{Riordan}
O.~Riordan, \emph{Spanning subgraphs of random graphs}, Combinatorics,
  Probability \& Computing \textbf{9} (2000), no.~2, 125--148.

\bibitem{Zhu25}
M.~Zhukovskii, \emph{Sharp thresholds for spanning regular subgraphs}, preprint
  available at \href{https://arxiv.org/abs/2502.14794}{arXiv:2502.14794}.

\end{thebibliography}

\makeatletter
\renewcommand\@biblabel[1]{#1.}
\makeatother

\providecommand{\bysame}{\leavevmode\hbox to3em{\hrulefill}\thinspace}
\providecommand{\MR}{\relax\ifhmode\unskip\space\fi MR }
\providecommand{\MRhref}[2]{%
  \href{http://www.ams.org/mathscinet-getitem?mr=#1}{#2}
}
\providecommand{\href}[2]{#2}

\end{document}